\documentclass[10pt,reqno]{amsart}

\usepackage{amsmath}
\usepackage{amssymb}
\usepackage{amscd}
\usepackage{hhline}
\usepackage{dcpic}
\usepackage{pictex}

\newtheorem{theorem}{Theorem}

\newtheorem{theoremb}{Theorem}
\newtheorem{theoremc}{Theorem}

\newtheorem{theoreme}{Theorem}

\newtheorem{rk}[theoremc]{Remark}

\newtheorem{lem}[theoreme]{Lemma}

\newtheorem{prop}[theoremb]{Proposition}
\newcommand\bib[1]{\bibitem[#1]{#1}}

\newcommand\com[1]{}

\newcommand\C{{\mathbb C}}

\newcommand\E{\mathcal{E}}

\newcommand\La{\Lambda}

\newcommand\oo{\omega}
\newcommand\op[1]{\mathop{\rm #1}\nolimits}
\newcommand\ot{\otimes}
\newcommand\p{\partial}

\newcommand\R{{\mathbb R}}

\newcommand\z{\sigma}

\begin{document}

 \title{Involutivity of field equations}
 \author{Boris Kruglikov}
 \date{}
 \address{Institute of Mathematics and Statistics, University of Troms\o, Troms\o\ 90-37, Norway.}
 \email{boris.kruglikov@uit.no}

 \vspace{-14.5pt}
 \begin{abstract}
We prove involutivity of Einstein, Einstein-Maxwell and other field equations
by calculating the Spencer cohomology of these systems. Relation with
Cartan method is traced in details. Basic implications through
Cartan-K\"ahler theory are derived.
 \end{abstract}

 \maketitle

\section*{Introduction}

The field equations, derived independently by A.Einstein \cite{E}
and D.Hilbert \cite{H}, form the basis of the general relativity.
Soon after these equations had appeared in 1915 (see \cite{LMP,T} for the history of invention)
some very important particular solutions were found,
yet not much has been known on their solution space before E.Cartan's contribution.

In correspondence with A.Einstein \cite{CE} he
established involutivity of various field equations.
Cartan approached this through his theory of exterior differential systems \cite{C$_1$},
namely by calculating Cartan characters and verifying that they pass the Cartan test.
This is quite an involved work. Detailed proof of this is contained in \cite{Ga}
(see also \cite{DT}; quite a different proof of involutivity for Einstein, Yang-Mills and other equations
is contained in \cite{DV}).

In this paper we prove involutivity of the field equations using the
formal theory of differential equations \cite{S}: we calculate all Spencer
cohomology of the system and check their vanishing in the prescribed
range (together with vanishing of the structure tensor).
By Serre's contribution to \cite{GS} this is equivalent to Cartan test.
Since chasing diagrams is considered nowadays standard, this
turns out to be a reasonable path.

Let us notice that up to now the Spencer $\delta$-cohomology has never
been evaluated for the field equations.
We do our calculations for both Einstein and Einstein-Maxwell
system, and also encompass such systems as pure radiation and dust.
We start with vacuum or source-free equations and then impose additional fields,
proving involutivity of various field equations.

We relate our calculations to those of Cartan
which is not obvious, since the two theories -- Cartan and Spencer
-- though accepted being equivalent, are not in direct
correspondence.
Finally we derive some simple but important
implications using the Cartan-K\"ahler theorem.

Let us mention that in his papers \cite{C$_2$} Cartan mostly
considers the so-called unified field theory based on distant
parallelism\footnote{Specific references are vol.II p. 1199--1229
and vol.III-1 p.549--611.}, which corresponds to Einstein 
system with non-zero torsion, so that the number of differential
equations in the system is $22$, not $10$ or $18$ (or $14$) as in
the usual Einstein and Einstein-Maxwell equations. Involutivity of
these two latter do not follow from involutivity of the former upon a
specification. In addition, Cartan arguments by exhibiting relations
between the equations but not proving they are all. The formal theory approach,
adapted here, provides both rigorous and economic way to
prove involutivity.

\section{Background: jets, Spencer cohomology and all this}

We will consider here only the theory of systems of PDEs of the same order $k$ \cite{Go,S}.
The general theory, developed in \cite{KL$_1$,KL$_3$}, shall be useful
for other purposes.

Thus let $\E=\E_k\subset J^k\pi$ be a submanifold in the space of $k$-jets of sections
of a bundle $\pi:E\to M$, subject to certain regularity assumptions, which
include the claim that $\pi_{k,k-1}:\E_k\to J^{k-1}\pi$ is submersion.
We let $\E_l=J^l\pi$ for $l<k$ and $\E_l=\E_k^{(l-k)}$ for $l>k$,
where the latter space is the prolongation defined as
 $$
\E_k^{(l-k)}=\{[s]_x^l\in J^l\pi\,:\text{ jet-prolongation }j^k(s)
\text{ is tangent to }\E_k\text{ at }[s]_x^k\text{ with order }(l-k)\},
 $$
where $[s]_x^k$ is the $k$-jet of the (local) section $s\in\Gamma(\pi)$.
Equation $\E$ is called formally integrable if all the projections $\pi_{l,l-1}:\E_l\to\E_{l-1}$ are submersions.

Let us denote by $N$ the tangent space to the fiber of $\pi$ and by $T$ the tangent space to $M$. Then the symbol spaces $g_t\subset S^tT^*\ot N$ are the kernels of $d\pi_{t,t-1}:T\E_t\to T\E_{t-1}$. We obviously have $g_t=S^tT^*\ot N$ for $t<k$ and
the space $g_k$ is determined by the equation, however the higher index spaces are difficult to calculate without knowledge of formal integrability.

Instead one considers formal prolongations defined as
$g_t=(g_k\ot S^{t-k}T^*)\cap(S^tT^*\ot N)$ for $t>k$. These symbols are united into the Spencer $\delta$-complex
 \begin{equation}\label{Spencer}
0\to g_t\to g_{t-1}\ot T^*\stackrel{\delta}\longrightarrow
g_{t-2}\ot\La^2T^*\stackrel{\delta}\longrightarrow\dots\stackrel{\delta}\longrightarrow
g_{t-i}\ot\La^iT^*\to\dots
 \end{equation}
with morphisms $\delta$ being the symbols of the de Rham operator. The cohomology at
the term $g_{t-i}\ot\La^iT^*$ is denoted by $H^{t-i,i}(\E)$ and is called Spencer
$\delta$-cohomology of $\E$.

Formal theory of PDEs describes obstructions to formal integrability
as elements $W_t\in H^{t-1,2}(\E)$, called curvature, torsion, structure functions or
Weyl tensors. Their vanishing is equivalent to formal integrability
(and in certain cases to local integrability).

Symbolic system $g=\oplus g_t$ is called involutive if $H^{i,j}(g)=0$ for all $i\ne k-1$
and $i+j>0$. This is equivalent to fulfillment of Cartan test for the corresponding EDS
(which in turn means a PDE system of the 1st order).

Equation $\E$ is called involutive if its symbolic system is involutive and in addition
the only obstruction $W_k$ vanishes. Thus involutive systems are formally integrable.

Advantage of involutive systems is that compatibility conditions should be calculated only
at one order, while in general they exist in different places and one shall carry the whole
prolongation-projection method through \cite{S,KLV,KL$_3$}. Fortunately many equations
of mathematical physics are involutive
(this can be easily checked for all determined and underdetermined
equations\footnote{An equation is (under)determined if the symbol of its defining operator
is an isomorphism (resp. surjective).})
and we are going to prove this for relativity equations.

\section{Einstein equations}\label{S2}

We run the setup very briefly, referring to plentiful books on
differential geometry and relativity for details, e.g. \cite{B,K,St}.

Let $M$ be a (4-dimensional in physical applications) manifold, $g$ pseudo-Riemannian
tensor (for relativity: of Lorentz signature (1,3)) with Ricci
tensor $\op{Ric}$ and scalar curvature $R$, $\Lambda$ a cosmological
constant and $\mathcal{T}$ the energy-momenta tensor. The Einstein equations \cite{E,H}
are:
 \begin{equation}\label{E}
\op{Ric}-\tfrac12R\,g+\Lambda\,g=\mathcal{T}.
 \end{equation}
We will assume at first that $\mathcal{T}$ is a given tensor
(in some models like electromagnetic it is traceless, which implies that
the scalar curvature $R$ is constant, but in general it is not so),
and only the metric $g$ is unknown
(further on we'll treat the case, when $\mathcal{T}$ depends on unknowns fields
entering the equations). We can re-write (\ref{E}) as
 $$
G(\op{Ric})=\mathcal{T}_\Lambda,
 $$
where $G(h)=h-\frac12\op{Tr}_g(h)\,g$, $h\in S^2T^*$, is the gravitational
operator\footnote{Here and below $T=TM$ is the tangent bundle to $M$
and $T^*$ is the cotangent bundle. $S^2_+T^*\subset S^2T^*$ is the bundle of non-degenerate quadrics
(not necessary positive definite).}
and $\mathcal{T}_\Lambda=\mathcal{T}-\Lambda\,g$.

Bianchi identity reads
 \begin{equation}\label{div}
\mathfrak{d}(\op{Ric})=0,\ \text{ where }\ \mathfrak{d}=\op{div}_g\circ\,G
 \end{equation}
and $\op{div}_g$ is the divergence operator on symmetric tensors, so that
$\mathfrak{d}h(\xi)=\op{Tr}_g(\nabla_{\cdot}h(\cdot,\xi))-\frac12\nabla_\xi(\op{Tr}_gh)$
for $\xi\in T$. This implies the following conservation law:
 \begin{equation}\label{divT}
\op{div}_g(\mathcal{T}_\Lambda)=\op{div}_g(\mathcal{T})=0.
 \end{equation}
This is the first order PDE w.r.t. $g$ and so system (\ref{E}) is not involutive unless
$\mathcal{T}=0$ (indeed, we get a non-trivial equation of lower order than
(\ref{E}), see Appendix for more details on this).

Thus in what follows in this section we'll concentrate on the vacuum
case\footnote{DeTurck's idea \cite{B} is to use covariance of the
left hand side $G[g]$ of (\ref{E}) and change the equation to
$G[g]=\varphi^*\mathcal{T}$, where $\varphi:M\to M$ is a diffeomorphism, so
that $\mathcal{T}$ is given while $(g,\varphi)$ unknown. This system (coupled with compatibility
$\delta_{\varphi\!_{_*}g}\mathcal{T}=0$) is already involutive for any non-degenerate $\mathcal{T}$
-- see Section \ref{S5}. DeTruck \cite{DT} proves solvability of (\ref{E}) differently --
by establishing involutivity of its prolongation-projection (\ref{E})+(\ref{divT}). \label{DeT}}:
$\mathcal{T}=0$.

Tracing (\ref{E}) by $g$ yields $4\Lambda=R=\op{const}$, so that the Einstein
equation $\E$ is equivalent to
 \begin{equation}\label{Ev}
\op{Ric}=\Lambda\,g.
 \end{equation}
To understand this equation we need to study solvability of
the Ricci operator
 $$
\op{Ric}:C^\infty(S^2_+T^*)\to C^\infty(S^2T^*),
 $$
which gives rise to the sequence of operators
$\phi_{\op{Ric}}:J^{k+2}(S^2_+T^*)\to J^k(S^2T^*)$ with symbols
$\z^{(k)}_{\op{Ric}}:S^{k+2}T^*\ot S^2T^*\to S^kT^*\ot S^2T^*$.
The symbol of the Ricci operator was described in \cite{DT,Ga,B}
and it equals (we choose $p,q,\dots{}\in T^*$; it suffices to define the map on decomposable quadrics)
 $$
\z_{\op{Ric}}(p^2\otimes q^2)=\langle p,q\rangle_g\,p\,q-\tfrac12(|p|^2_gq^2+|q|^2_gp^2),
 $$
which is clearly an epimorphic map for any signature of $g$.
The prolongation
 $$
\z^{(k)}_{\op{Ric}}(\tfrac1{k+2}p^{k+2}\otimes q^2)= \langle p,q\rangle_g\,p^k\otimes p\,q-
\tfrac12(|p|^2_g\,p^k\ot q^2+|q|^2_g\,p^k\ot p^2)
 $$
is not epimorphic for $k>0$ as follows from the Bianchi identity.

Symbols of the Einstein equations (\ref{Ev}) are precisely
 $$
g_{k+2}=\op{Ker}(\z^{(k)}_{\op{Ric}}),\ k\ge0,
 $$
and we let $g_0=S^2T^*$, $g_1=T^*\ot S^2T^*$.

We calculate the Spencer cohomology of $\E$ (\ref{Ev})
by constructing resolutions to the symbols of the Ricci operator.
The first Spencer complex is exact.%
\com{
 $$
\begin{array}{ccccccccc}
 && 0 && 0 && 0\\
 && \downarrow && \downarrow && \downarrow\\
 0 & \to & g_2 & \longrightarrow & g_1\ot T^* & \longrightarrow & g_0\ot\La^2T^* & \to 0\\
 && \downarrow && \downarrow && \downarrow\\
 0 & \to & S^2T^*\ot S^2T^* & \to & T^*\ot S^2T^*\ot T^* &
\to & S^2T^*\ot\La^2T^* & \to 0\\
 && \downarrow && \downarrow && \downarrow\\
 0 & \to & S^2T^* & \longrightarrow & 0 && 0\\
 && \downarrow \\
 && 0
\end{array}
 $$
 }
The second Spencer complex is included into the commutative diagram,
implying $H^{1,1}(\E)=S^2T^*$ and $H^{2-i,i}(\E)=0$ for $i\ne1$:

 $$
\begindc{\commdiag}[3]
 \obj(20,50)[12]{$0$}
 \obj(50,50)[13]{$0$}
 \obj(80,50)[14]{$0$}
 \obj(0,40)[21]{$0$}
 \obj(20,40)[22]{\ $g_2$\ }
 \obj(50,40)[23]{\ $g_1\ot T^*$\ }
 \obj(80,40)[24]{\ $g_0\ot\La^2T^*$\ }
 \obj(100,40)[25]{$0$}
 \obj(0,30)[31]{$0$}
 \obj(20,30)[32]{\ $S^2T^*\ot S^2T^*$\ }
 \obj(50,30)[33]{\ $T^*\ot S^2T^*\ot T^*$\ }
 \obj(80,30)[34]{\ $S^2T^*\ot\La^2T^*$\ }
 \obj(100,30)[35]{$0$}
 \obj(0,20)[41]{$0$}
 \obj(20,20)[42]{\ $S^2T^*$\ }
 \obj(50,20)[43]{$0$}
 \obj(80,20)[44]{$0$}
 \obj(20,10)[52]{$0$}
 \mor{12}{22}{} \mor{13}{23}{}  \mor{14}{24}{}
 \mor{22}{32}{} \mor{23}{33}{} \mor{24}{34}{}
 \mor{32}{42}{\footnotesize$\z_\text{Ric}$\,}[\atright,\solidarrow]
  \mor{33}{43}{} \mor{34}{44}{}
 \mor{42}{52}{}
 \mor{21}{22}{} \mor{22}{23}{} \mor{23}{24}{} \mor{24}{25}{}
 \mor{31}{32}{} \mor{32}{33}{} \mor{33}{34}{} \mor{34}{35}{}
 \mor{41}{42}{} \mor{42}{43}{}
 \mor(20,21)(46,39){}[\atright, \dasharrow]
\enddc
 $$
In what follows we shorten $S^kT^*$ to $S^k$, and use similar
notation $\La^k=\La^kT^*$ for readability of the diagrams. The third Spencer
complex is included into the commutative diagram, implying
$H^{1,2}(\E)=T^*$ and $H^{3-i,i}(\E)=0$ for $i\ne2$:

 $$
\begindc{\commdiag}[3]
 \obj(15,50)[12]{$0$}
 \obj(40,50)[13]{$0$}
 \obj(65,50)[14]{$0$}
 \obj(90,50)[15]{$0$}
 \obj(0,40)[21]{$0$}
 \obj(15,40)[22]{\ $g_3$\ }
 \obj(40,40)[23]{\ $g_2\ot T^*$\ }
 \obj(65,40)[24]{\ $g_1\ot\La^2$\ }
 \obj(90,40)[25]{\ $g_0\ot\La^3$\ }
 \obj(105,40)[26]{$0$}
 \obj(0,30)[31]{$0$}
 \obj(15,30)[32]{\ $S^3\ot S^2$\ }
 \obj(40,30)[33]{\ $S^2\ot S^2\ot T^*$}
 \obj(65,30)[34]{$T^*\ot S^2\ot\La^2$}
 \obj(90,30)[35]{\ $S^2\ot\La^3$\ }
 \obj(105,30)[36]{$0$}
 \obj(0,20)[41]{$0$}
 \obj(15,20)[42]{\ $T^*\ot S^2$\ }
 \obj(40,20)[43]{\ $S^2\ot T^*$\ }
 \obj(65,20)[44]{$0$}
 \obj(90,20)[45]{$0$}
 \obj(0,10)[51]{$0$}
 \obj(15,10)[52]{\ $T^*$\ }
 \obj(40,10)[53]{$0$}
 \obj(15,0)[62]{$0$}
 \mor{12}{22}{} \mor{13}{23}{} \mor{14}{24}{} \mor{15}{25}{}
 \mor{22}{32}{} \mor{23}{33}{} \mor{24}{34}{} \mor{25}{35}{}
 \mor{32}{42}{\footnotesize$\z_\text{Ric}^{(1)}$\,}[\atright,\solidarrow]
  \mor{33}{43}{} \mor{34}{44}{} \mor{35}{45}{}
 \mor{42}{52}{\footnotesize$\z_{\mathfrak{d}}$\,}[\atright,\solidarrow]
  \mor{43}{53}{}
 \mor{52}{62}{}
 \mor{21}{22}{} \mor{22}{23}{} \mor{23}{24}{} \mor{24}{25}{} \mor{25}{26}{}
 \mor{31}{32}{} \mor{32}{33}{} \mor{33}{34}{} \mor{34}{35}{} \mor{35}{36}{}
 \mor{41}{42}{} \mor{42}{43}{} \mor{43}{44}{}
 \mor{51}{52}{} \mor{52}{53}{}
 \mor(17,11)(32,21){}[\atright, \dashline]
 \mor(29,19)(44,29){}[\atright, \dashline]
 \mor(47,31)(59,39){}[\atright, \dasharrow]
\enddc
 $$
The formula for the symbol of operator (\ref{div}) $\mathfrak{d}:C^\infty(S^2T^*)\to C^\infty(T^*)$,
 $$
\z_{\mathfrak{d}}(p\otimes q\cdot r)=\tfrac12\bigl(\langle p,q\rangle_g\,r+\langle p,r\rangle_g\,q-\langle q,r\rangle_g\,p\bigr)
 $$easily implies exactness of the first column.
The next commutative diagram is exact:

 $$
\begindc{\commdiag}[3]
 \obj(15,50)[12]{$0$}
 \obj(40,50)[13]{$0$}
 \obj(65,50)[14]{$0$}
 \obj(90,50)[15]{$0$}
 \obj(115,50)[16]{$0$}
 \obj(0,40)[21]{$0$}
 \obj(15,40)[22]{\ $g_4$\ }
 \obj(40,40)[23]{\ $g_3\ot T^*$\ }
 \obj(65,40)[24]{\ $g_2\ot\La^2$\ }
 \obj(90,40)[25]{\ $g_1\ot\La^3$\ }
 \obj(115,40)[26]{\ $g_0\ot\La^4$\ }
 \obj(130,40)[27]{$0$}
 \obj(0,30)[31]{$0$}
 \obj(15,30)[32]{\ $S^4\ot S^2$\ }
 \obj(40,30)[33]{\ $S^3\ot S^2\ot T^*$}
 \obj(65,30)[34]{$S^2\ot S^2\ot\La^2$}
 \obj(90,30)[35]{$T^*\ot S^2\ot\La^3$\ }
 \obj(115,30)[36]{\ $S^2\ot\La^4$\ }
 \obj(130,30)[37]{$0$}
 \obj(0,20)[41]{$0$}
 \obj(15,20)[42]{\ $S^2\ot S^2$\ }
 \obj(40,20)[43]{\ $T^*\ot S^2\ot T^*$\ }
 \obj(65,20)[44]{\ $S^2\ot\La^2$\ }
 \obj(90,20)[45]{$0$}
 \obj(115,20)[46]{$0$}
 \obj(0,10)[51]{$0$}
 \obj(15,10)[52]{\ $T^*\ot T^*$\ }
 \obj(40,10)[53]{\ $T^*\ot T^*$\ }
 \obj(65,10)[54]{$0$}
 \obj(15,0)[62]{$0$}
 \obj(40,0)[63]{$0$}
 \mor{12}{22}{} \mor{13}{23}{} \mor{14}{24}{} \mor{15}{25}{} \mor{16}{26}{}
 \mor{22}{32}{} \mor{23}{33}{} \mor{24}{34}{} \mor{25}{35}{} \mor{26}{36}{}
 \mor{32}{42}{\footnotesize$\z_\text{Ric}^{(2)}$\,}[\atright,\solidarrow]
  \mor{33}{43}{} \mor{34}{44}{} \mor{35}{45}{} \mor{36}{46}{}
 \mor{42}{52}{\footnotesize$\z_{\mathfrak{d}}^{(1)}$\,}[\atright,\solidarrow]
  \mor{43}{53}{} \mor{44}{54}{}
 \mor{52}{62}{} \mor{53}{63}{}
 \mor{21}{22}{} \mor{22}{23}{} \mor{23}{24}{} \mor{24}{25}{} \mor{25}{26}{} \mor{26}{27}{}
 \mor{31}{32}{} \mor{32}{33}{} \mor{33}{34}{} \mor{34}{35}{} \mor{35}{36}{} \mor{36}{37}{}
 \mor{41}{42}{} \mor{42}{43}{} \mor{43}{44}{} \mor{44}{45}{}
 \mor{51}{52}{} \mor{52}{53}{} \mor{53}{54}{}
\enddc
 $$
and this extends to the commutative diagram for any $k\ge0$, with
exact rows and columns (we draw the diagram for the case $n=4$, but this has an obvious extension):

 $$
\begindc{\commdiag}[3]
 \obj(14,50)[12]{$0$}
 \obj(38,50)[13]{$0$}
 \obj(65,50)[14]{$0$}
 \obj(92,50)[15]{$0$}
 \obj(119,50)[16]{$0$}
 \obj(0,40)[21]{$0$}
 \obj(14,40)[22]{$g_{k+4}$}
 \obj(38,40)[23]{$g_{k+3}\ot T^*$}
 \obj(65,40)[24]{$g_{k+2}\ot\La^2$}
 \obj(92,40)[25]{$g_{k+1}\ot\La^3$}
 \obj(119,40)[26]{$g_k\ot\La^4$}
 \obj(137,40)[27]{\!$0$}
 \obj(0,30)[31]{$0$}
 \obj(14,30)[32]{$S^{k+4}\ot S^2$}
 \obj(38,30)[33]{$S^{k+3}\ot S^2\ot T^*$}
 \obj(65,30)[34]{$S^{k+2}\ot S^2\ot\La^2$}
 \obj(92,30)[35]{$S^{k+1}\ot S^2\ot\La^3$}
 \obj(119,30)[36]{$S^k\ot S^2\ot\La^4$}
 \obj(137,30)[37]{\!$0$}
 \obj(0,20)[41]{$0$}
 \obj(14,20)[42]{$S^{k+2}\ot S^2$}
 \obj(38,20)[43]{$S^{k+1}\ot S^2\ot T^*$}
 \obj(65,20)[44]{$S^k\ot S^2\ot\La^2$}
 \obj(92,20)[45]{$S^{k-1}\ot S^2\ot\La^3$}
 \obj(119,20)[46]{$S^{k-2}\ot S^2\ot\La^4$}
 \obj(137,20)[47]{$0$}
 \obj(0,10)[51]{$0$}
 \obj(15,10)[52]{$S^{k+1}\ot T^*$}
 \obj(38,10)[53]{$S^k\ot T^*\ot T^*$}
 \obj(65,10)[54]{$S^{k-1}\ot T^*\ot\La^2$}
 \obj(92,10)[55]{$S^{k-2}\ot T^*\ot\La^3$}
 \obj(119,10)[56]{$S^{k-3}\ot T^*\ot\La^4$}
 \obj(137,10)[57]{$0$}
 \obj(15,0)[62]{$0$}
 \obj(38,0)[63]{$0$}
 \obj(65,0)[64]{$0$}
 \obj(92,0)[65]{$0$}
 \obj(119,0)[66]{$0$}
 \mor{12}{22}{} \mor{13}{23}{} \mor{14}{24}{} \mor{15}{25}{} \mor{16}{26}{}
 \mor{22}{32}{} \mor{23}{33}{} \mor{24}{34}{} \mor{25}{35}{} \mor{26}{36}{}
 \mor{32}{42}{\footnotesize$\z_\text{Ric}^{(k+2)}$\,}[\atright,\solidarrow]
  \mor{33}{43}{} \mor{34}{44}{} \mor{35}{45}{} \mor{36}{46}{}
 \mor{42}{52}{}[\atright,\solidarrow]
 \obj(9,15)[]{{\footnotesize${\z^{(k+1)}_{\mathfrak{d}}}$\!}}
 \mor{43}{53}{} \mor{44}{54}{} \mor{45}{55}{} \mor{46}{56}{}
 \mor{52}{62}{} \mor{53}{63}{} \mor{54}{64}{} \mor{55}{65}{} \mor{56}{66}{}
 \mor{21}{22}{} \mor{22}{23}{} \mor{23}{24}{} \mor{24}{25}{} \mor{25}{26}{} \mor{26}{27}{}
 \mor{31}{32}{} \mor{32}{33}{} \mor{33}{34}{} \mor{34}{35}{} \mor{35}{36}{} \mor{36}{37}{}
 \mor{41}{42}{} \mor{42}{43}{} \mor{43}{44}{} \mor{44}{45}{} \mor{45}{46}{} \mor{46}{47}{}
 \mor{51}{52}{} \mor{52}{53}{} \mor{53}{54}{} \mor{54}{55}{} \mor{55}{56}{} \mor{56}{57}{}
\enddc
 $$

Since $H^{k,l}(\E)=0$ for $k\ge2$ and all $l$, the symbolic system
$g=\oplus g_k$ of $\E$ is involutive. To prove involutivity of PDE
system $\E$ it is thus enough to check that the actual compatibility
conditions belonging to $H^{1,2}(\E)$ vanish. But these elements, as
follows from calculation of their symbols, coincide with Bianchi
relations and thus their equality to zero holds identically.

Notice that we don't use specific features of Lorentz geometry, and our arguments work
generally (though the diagrams become larger). We have proved:

 \begin{theorem}
The only nonzero Spencer $\delta$-cohomology of Einstein equation (\ref{E})$_{\mathcal{T}=0}$
in any dimension and signature are
 $$
H^{0,0}(\E)\simeq S^2T^*,\ H^{1,1}(\E)\simeq S^2T^*,\ H^{1,2}(\E)\simeq T^*.
 $$
The Einstein equation $\E$ is involutive.
 \end{theorem}

\section{Relation with Cartan approach}

In Cartan theory involutivity is checked via Cartan characters, which are defined
as follows. Consider a symbolic system $g$ of order $k$:
 $$
g_i=S^iT^*\ot N,\ i<k;\qquad g_k\subset S^kT^*\ot N;\qquad
g_i=g_k^{(i-k)}\subset S^iT^*\ot N,\ i>k.
 $$
Let $0=V_n^*\subset V_{n-1}^*\subset\dots\subset V_0^*=T^*$ be a generic complete flag
(difference of dimensions is 1; $n=\dim T=\dim M$). By definition
 $$
s_i=\dim(g_k\cap S^kV^*_{i-1}\ot N)-\dim(g_k\cap S^kV^*_i\ot N),\quad 1\le i\le n
 $$
(the sequence monotonically decreases), so that $\dim g_k=s_1+\dots+s_n$.

Cartan test \cite{C$_2$,Ma,BCG$^3$} claims that symbolic system $g$ is involutive iff
 $$
\dim g_{k+1}=s_1+2s_2+\dots+ns_n.
 $$
In this case we also have $\dim g_{k+1}=s_1'+\dots+s_n'$, where $s_i'=s_i+\dots+s_n$
are the Cartan characters for the prolongation, i.e.
 $$
s_i'=\dim(g_{k+1}\cap S^{k+1}V^*_{i-1}\ot N)-\dim(g_{k+1}\cap S^{k+1}V^*_i\ot N),
\quad 1\le i\le n.
 $$
Thus we can calculate the dimensions of symbol spaces via Cartan characters:
 \begin{equation}\label{dimg}
\dim g_l=m\binom{n+l-1}l,\ l<k,\qquad
\dim g_l=\sum_{i=1}^n\binom{l-k+i-1}{i-1}s_i,\ i\ge k,
 \end{equation}
where $m=\dim N$, which we also denote by $s_0$.

Let us relate Cartan characters and Spencer $\delta$-cohomology for an involutive system
of pure order $k$. The only nontrivial dimensions of the latter are
$h_0=\dim H^{0,0}=m$ and $h_i=\dim H^{k-1,i}$, $1\le i\le n$.
To one side the relation is given by

 \begin{prop}\label{prop1}
The numbers $(h_0,\dots,h_n)$ are expressed through $(s_0,\dots,s_n)$ as: $h_0=s_0$ and
 $$
h_l=(-1)^l\sum_{j=1}^ns_j\sum_{i=0}^{l-1}(-1)^i\binom{n}i\binom{l+j-i-2}{j-1}
+(-1)^ls_0\sum_{i=l}^n(-1)^i\binom{n}i\binom{k+l+n-i-2}{n-1}
 $$
for $\quad 1\le i\le n$.
 \end{prop}

 \begin{proof}
Due to involutivity the Euler characteristic of Spencer complex (\ref{Spencer})
equals $(-1)^{t-k+1}h_{t-k+1}$ for $t\ge k$, zero for $0<t<k$ and $h_0$ for $t=0$.
Calculating it directly as $\sum_{i=0}^n(-1)^i\dim g_{t-i}\binom{n}i$ and using
(\ref{dimg}) we get the result.
 \end{proof}

The relations above are invertible, but we obtain the inverse formula from
another idea.

 \begin{prop}\label{prop2}
The numbers $(s_0,\dots,s_n)$ are expressed through $(h_0,\dots,h_n)$
in triangular way: $s_0=h_0$ and
 $$
s_l=\binom{n+k-l-1}{k-1}h_0+\sum_{i=n-l+1}^n(-1)^{n-l-i}\binom{i-1}{n-l}h_i,\qquad
l=1,\dots,n.
 $$
 \end{prop}

 \begin{proof}
For $l\gg1$ the expression $H_\E(l)=\sum_{i\le l}\dim g_i$ is a polynomial, called Hilbert polynomial of $g$ and so $\dim g_z=H_\E(z)-H_\E(z-1)$ is a polynomial too (for large
integers $z=l$ and we extend it to the space of all $z\in\C$).

The Hilbert polynomial can be computed through the standard resolution of the
symbolic module $g^*$ \cite{Gr,KL$_2$} and we get:
 \begin{multline}\label{Hdg}
\dim g_z=\sum_{i,j}(-1)^j\dim H^{i,j}(g)\cdot\binom{z+n-i-j-1}{n-1}\\
=h_0\binom{z+n-1}{n-1}-h_1\binom{z+n-k-1}{n-1}+h_2\binom{z+n-k-2}{n-1}-\dots
 \end{multline}
On the other hand from (\ref{dimg}) we have the following expression:
 \begin{equation}\label{Sdg}
\dim g_z=\sum_{l=1}^n s_l\binom{z+l-k-1}{l-1}.
 \end{equation}
Comparing (\ref{Hdg}) to (\ref{Sdg}) we obtain the result: At first substitute
$z=k-1$ and get\footnote{One shall be careful: in this substitution $\dim g_z$
is understood as analytic continuation (\ref{Sdg}), because the actual value
of $\dim g_{k-1}$ could be different; on the other hand studying the large values
$l$ one gets the same result.}
 $$
s_1=h_0\binom{n+k-2}{n-1}-h_n,
 $$
then calculate difference derivative by $z$, substitute $z=k-2$ and get the formula
 $$
s_2=h_0\binom{n+k-3}{n-2}-h_{n-1}+h_n\binom{n-1}{n-2}
 $$
and so on.
 \end{proof}

 \begin{rk}
To see that formulae of proposition \ref{prop2} invert these of proposition \ref{prop1}
is not completely trivial: one must use certain combinatorial identities.
 \end{rk}

Now let us apply the result to Einstein vacuum equations (we restrict to
the physical dimension $n=4$, but due to previous formulae the general case is
easily restored). As we calculated in the previous section
 $$
h_0=10,\ h_1=10,\ h_2=4,\ h_3=h_4=0.
 $$
Thus proposition \ref{prop2} implies that the Cartan characters are
 $$
s_1=40,\ s_2=30,\ s_3=16,\ s_4=4.
 $$
This calculation was independently verified in \textsc{Maple}
\textsf{with(DifferentialGeometry):} \cite{A}.

In particular,
the Cartan genre is 4 and the Cartan integer is 4, i.e. the general (analytic)
solution of the Einstein vacuum equations depends on 4 functions of 4 arguments.
This is indeed so due to covariance: the group $\op{Diff}_\text{loc}(M)$ acts
on $\E$ as symmetries.

We can calculate the Hilbert polynomial of the Einstein equation
 $$
H_\E(z)=10+22z+\frac{89z^2}6+3z^3+\frac{z^4}6.
 $$
The first dimensions of the symbol spaces are:
 $$
\dim g_0=10,\ \dim g_1=40,\ \dim g_2=90,\ \dim g_3=164,\ \dim g_4=266,\ \dim g_5=400,\dots
 $$
(this in particular shows that direct calculation can be costly). The Cartan test
works as follows:
 $$
s_1+2s_2+3s_3+4s_4=164=\dim g_3.
 $$

\medskip

{\it Historical remark.\/}
Explicit formulae for Hilbert function and polynomial of PDE systems were calculated
already by Janet \cite{J} (and maybe to some extent were known to Riquier).

Nowadays his result is generalized to arbitrary involutive bases
and this provides a link between Cartan characters and Spencer cohomology.
The first step -- to read off the Hilbert function from an involutive basis
(or Pommaret basis or Cartan characters) -- can be found in a number of papers \cite{Se,Ap,GB,PR}.

The second step comes from commutative algebra, where it is well-known that the Koszul homology
and the Hilbert polynomial are related (beware
that involutive systems are more important
in PDE context and are infrequently discussed in algebraic situations). Since Spencer cohomology
are $\R$-dual to Koszul homology, this yields a principal link, mentioned above.

This link however is not clearly marked in the literature, and the above explicit formulae are new.
In addition we deduce relations in both direction without implicit inversion
(involving lots of combinatorics).

\section{Einstein-Maxwell equations}\label{S4}

These equation extends (\ref{E}) in the sense that the
energy-momenta tensor is prescribed as electromagnetic tensor.
Denote by $J$ the current density. Einstein-Maxwell equations have
the following form\footnote{We set the cosmological constant
$\Lambda=0$, which does not restrict mathematics (can be
incorporated back without destroying any conclusion), but agrees
with physical observations.}:
 \begin{equation}\label{EM1}
\op{Ric}-\tfrac12R\,g=(F^2)_0,\quad dF=0,\quad \delta_gF=J.
 \end{equation}
Here the tensor $F$ in the first equation is viewed as a
(1,1)-tensor (an operator field) via the metric, and
$(F^2)_0=F^2-\frac14\op{Tr}(F^2)$ is the traceless part of its
square, while $F$ in the latter equations is a 2-form, and
$\delta_g=\pm*\,d\,*:\Omega^2M\to\Omega^1M$ is the Hodge
codifferential\footnote{Not to be confused with
Spencer $\delta$-differential.}.

In order not to deal with involutivity of systems of PDEs of
different orders (the theory developed in \cite{KL$_1$}), we can
re-write the system as a pure 2nd order system by introducing the
potential $A\in\Omega^1M$, $F=dA$:
 \begin{equation}\label{EM2}
\op{Ric}-\tfrac12R\,g=(dA\circ dA)_0,\quad \delta_g(dA)=J.
 \end{equation}

Both systems (\ref{EM1}) and (\ref{EM2}) have the following
compatibility condition of order 1 in $g$: $\delta_gJ=0$. Thus they
are not involutive unless $J=0$. This will be assumed at the end of this section.

But let us study at first the pure Maxwell equation (with known $g$),
written as a 2nd order system with operator $\Pi=\delta_g\circ
d$:
 \begin{equation}\label{M2}
\Pi(A)=J.
 \end{equation}
The symbol of this operator equals
 $$
\z_\Pi^{ }:S^2T^*\ot T^*\to T^*,\qquad Q\ot p\mapsto p\lrcorner
Q-\op{Tr}_g(Q)p,
 $$
where in the first term to the right the dualization
$T^*\stackrel{g}\simeq T$ is used. Thus the symbol is epimorphic, while its
prolongations are not, since they have left divisor of zero:
 $$
\z_{\delta_g}^{(k-1)}\circ\z_\Pi^{(k)}=0\quad\text{for}\quad
\z_{\delta_g}:T^*\ot T^*\to\R,\quad q\ot p\mapsto \langle p,q\rangle_g.
 $$
The symbol of $\delta_g$ is however epimorphic together with all its
prolongations and so we get the sequence of commutative diagrams
with all rows and columns exact except for the top (Spencer
$\delta$-complex) and the bottom rows:
 $$
\begindc{\commdiag}[3]
 \obj(20,50)[12]{$0$}
 \obj(50,50)[13]{$0$}
 \obj(80,50)[14]{$0$}
 \obj(0,40)[21]{$0$}
 \obj(20,40)[22]{\ $g_2$\ }
 \obj(50,40)[23]{\ $g_1\ot T^*$\ }
 \obj(80,40)[24]{\ $g_0\ot\La^2T^*$\ }
 \obj(100,40)[25]{$0$}
 \obj(0,30)[31]{$0$}
 \obj(20,30)[32]{\ $S^2\ot T^*$\ }
 \obj(50,30)[33]{\ $T^*\ot T^*\ot T^*$\ }
 \obj(80,30)[34]{\ $T^*\ot\La^2$\ }
 \obj(100,30)[35]{$0$}
 \obj(0,20)[41]{$0$}
 \obj(20,20)[42]{\ $T^*$\ }
 \obj(50,20)[43]{$0$}
 \obj(80,20)[44]{$0$}
 \obj(20,10)[52]{$0$}
 \mor{12}{22}{} \mor{13}{23}{}  \mor{14}{24}{}
 \mor{22}{32}{} \mor{23}{33}{} \mor{24}{34}{}
 \mor{32}{42}{\footnotesize$\z_\Pi^{}$\,}[\atright,\solidarrow]
  \mor{33}{43}{} \mor{34}{44}{}
 \mor{42}{52}{}
 \mor{21}{22}{} \mor{22}{23}{} \mor{23}{24}{} \mor{24}{25}{}
 \mor{31}{32}{} \mor{32}{33}{} \mor{33}{34}{} \mor{34}{35}{}
 \mor{41}{42}{} \mor{42}{43}{}
 \mor(20,21)(46,39){}[\atright, \dasharrow]
\enddc
 $$
This implies $H^{1,1}\simeq T^*$. The next complex
 $$
\begindc{\commdiag}[3]
 \obj(15,50)[12]{$0$}
 \obj(40,50)[13]{$0$}
 \obj(65,50)[14]{$0$}
 \obj(90,50)[15]{$0$}
 \obj(0,40)[21]{$0$}
 \obj(15,40)[22]{\ $g_3$\ }
 \obj(40,40)[23]{\ $g_2\ot T^*$\ }
 \obj(65,40)[24]{\ $g_1\ot\La^2$\ }
 \obj(90,40)[25]{\ $g_0\ot\La^3$\ }
 \obj(105,40)[26]{$0$}
 \obj(0,30)[31]{$0$}
 \obj(15,30)[32]{\ $S^3\ot T^*$\ }
 \obj(40,30)[33]{\ $S^2\ot T^*\ot T^*$}
 \obj(65,30)[34]{$T^*\ot T^*\ot\La^2$}
 \obj(90,30)[35]{\ $T^*\ot\La^3$\ }
 \obj(105,30)[36]{$0$}
 \obj(0,20)[41]{$0$}
 \obj(15,20)[42]{\ $T^*\ot T^*$\ }
 \obj(40,20)[43]{\ $T^*\ot T^*$\ }
 \obj(65,20)[44]{$0$}
 \obj(90,20)[45]{$0$}
 \obj(0,10)[51]{$0$}
 \obj(15,10)[52]{\ $\R$\ }
 \obj(40,10)[53]{$0$}
 \obj(15,0)[62]{$0$}
 \mor{12}{22}{} \mor{13}{23}{} \mor{14}{24}{} \mor{15}{25}{}
 \mor{22}{32}{} \mor{23}{33}{} \mor{24}{34}{} \mor{25}{35}{}
 \mor{32}{42}{\footnotesize$\z_\Pi^{(1)}$\,}[\atright,\solidarrow]
  \mor{33}{43}{} \mor{34}{44}{} \mor{35}{45}{}
 \mor{42}{52}{\footnotesize$\z_{\delta_g}$\,}[\atright,\solidarrow]
  \mor{43}{53}{}
 \mor{52}{62}{}
 \mor{21}{22}{} \mor{22}{23}{} \mor{23}{24}{} \mor{24}{25}{} \mor{25}{26}{}
 \mor{31}{32}{} \mor{32}{33}{} \mor{33}{34}{} \mor{34}{35}{} \mor{35}{36}{}
 \mor{41}{42}{} \mor{42}{43}{} \mor{43}{44}{}
 \mor{51}{52}{} \mor{52}{53}{}
 \mor(17,11)(32,21){}[\atright, \dashline]
 \mor(29,19)(44,29){}[\atright, \dashline]
 \mor(47,31)(59,39){}[\atright, \dasharrow]
\enddc
 $$
yields $H^{1,2}\simeq\R$. Further complexes are already exact.
Here's the next one:
 $$
\begindc{\commdiag}[3]
 \obj(15,50)[12]{$0$}
 \obj(40,50)[13]{$0$}
 \obj(65,50)[14]{$0$}
 \obj(90,50)[15]{$0$}
 \obj(115,50)[16]{$0$}
 \obj(0,40)[21]{$0$}
 \obj(15,40)[22]{\ $g_4$\ }
 \obj(40,40)[23]{\ $g_3\ot T^*$\ }
 \obj(65,40)[24]{\ $g_2\ot\La^2$\ }
 \obj(90,40)[25]{\ $g_1\ot\La^3$\ }
 \obj(115,40)[26]{\ $g_0\ot\La^4$\ }
 \obj(130,40)[27]{$0$}
 \obj(0,30)[31]{$0$}
 \obj(15,30)[32]{\ $S^4\ot T^*$\ }
 \obj(40,30)[33]{\ $S^3\ot T^*\ot T^*$}
 \obj(65,30)[34]{$S^2\ot T^*\ot\La^2$}
 \obj(90,30)[35]{$T^*\ot T^*\ot\La^3$\ }
 \obj(115,30)[36]{\ $T^*\ot\La^4$\ }
 \obj(130,30)[37]{$0$}
 \obj(0,20)[41]{$0$}
 \obj(15,20)[42]{\ $S^2\ot T^*$\ }
 \obj(40,20)[43]{\ $T^*\ot T^*\ot T^*$\ }
 \obj(65,20)[44]{\ $T^*\ot\La^2$\ }
 \obj(90,20)[45]{$0$}
 \obj(115,20)[46]{$0$}
 \obj(0,10)[51]{$0$}
 \obj(15,10)[52]{\ $T^*$\ }
 \obj(40,10)[53]{\ $T^*$\ }
 \obj(65,10)[54]{$0$}
 \obj(15,0)[62]{$0$}
 \obj(40,0)[63]{$0$}
 \mor{12}{22}{} \mor{13}{23}{} \mor{14}{24}{} \mor{15}{25}{} \mor{16}{26}{}
 \mor{22}{32}{} \mor{23}{33}{} \mor{24}{34}{} \mor{25}{35}{} \mor{26}{36}{}
 \mor{32}{42}{\footnotesize$\z_\Pi^{(2)}$\,}[\atright,\solidarrow]
  \mor{33}{43}{} \mor{34}{44}{} \mor{35}{45}{} \mor{36}{46}{}
 \mor{42}{52}{\footnotesize$\z_{\delta_g}^{(1)}$\,}[\atright,\solidarrow]
  \mor{43}{53}{} \mor{44}{54}{}
 \mor{52}{62}{} \mor{53}{63}{}
 \mor{21}{22}{} \mor{22}{23}{} \mor{23}{24}{} \mor{24}{25}{} \mor{25}{26}{} \mor{26}{27}{}
 \mor{31}{32}{} \mor{32}{33}{} \mor{33}{34}{} \mor{34}{35}{} \mor{35}{36}{} \mor{36}{37}{}
 \mor{41}{42}{} \mor{42}{43}{} \mor{43}{44}{} \mor{44}{45}{}
 \mor{51}{52}{} \mor{52}{53}{} \mor{53}{54}{}
\enddc
 $$
and one can easily prolong. Since $\dim H^{*,2}=1$, there is only
one compatibility condition and from its symbol we identify it with
the condition $\delta_gJ=0$ (which comes from the Hodge identity
$\delta_g^2=0$). Since the metric $g$ is fixed this implies:

 \begin{theorem}
The only nonzero Spencer $\delta$-cohomology of Maxwell equation
(\ref{M2})$_{J=0}$ in any dimension and signature are
 $$
H^{0,0}\simeq T^*,\ H^{1,1}\simeq T^*,\ H^{1,2}\simeq\R.
 $$
The Maxwell equation is involutive provided the compatibility $\delta_gJ=0$ holds.
 \end{theorem}

Now we can study Einstein-Maxwell equation. Energy-momentum tensor $\mathcal{T}=(F^2)_0$
satisfies
 $$
\op{div}_g\mathcal{T}=F(\cdot,J^\sharp),
 $$
where $J^\sharp$ is the $g$-lift of $J\in\Omega^1M$ (\cite{K}).
Therefore compatibility conditions for equation (\ref{EM1}) or (\ref{EM2})
are: $F(\cdot,J^\sharp)=0$ and $\delta_gJ=0$. They are trivial only for $J=0$,
and since they have lower order (in $g$ and $F$) this must hold for involutivity.

 \begin{rk}
It is tempted to change the Einstein-Maxwell system in the case $J\ne0$,
similar to DeTurck trick$^{\ref{DeT}}$, to the following:
 $$
\op{Ric}-\tfrac12R\,g=(F^2)_0,\quad dF=0,\quad \delta_gF=\phi^*\!J
 $$
for the unknown $(g,F,\phi)$. While the second compatibility condition
$\delta_g(\phi^*\!J)=0$ yields an underdetermined equation of 2nd order on $\phi$, the first condition
$F(\cdot,\phi^*\!J^\sharp)=0$ implies the 0-order condition $\det F=0$ on $F$, and so the system
cannot be involutive.
 \end{rk}

Thus we should restrict to the case of no external sources: $J=0$.
The key observation is that this system is weakly uncoupled, meaning
that its symbol splits into the sum of symbols of Einstein and
Maxwell equations. Thus Spencer cohomology becomes the direct sum,
and the compatibility condition for the system $\mathcal{EM}$
(\ref{EM2}) is the union of two respective compatibility conditions
(Bianchi and Hodge identities).

 \begin{theorem}
The only nonzero Spencer $\delta$-cohomology of source-free
Einstein-Maxwell equation (\ref{EM2})$_{J=0}$ in any dimension and
signature are
 $$
H^{0,0}(\mathcal{EM})\simeq S^2T^*\oplus T^*,\
H^{1,1}(\mathcal{EM})\simeq S^2T^*\oplus T^*,\
H^{1,2}(\mathcal{EM})\simeq T^*\oplus\R.
 $$
This Einstein-Maxwell equation $\mathcal{EM}$ is involutive.
 \end{theorem}

Couple of remark to properly place this results are of order.

 \begin{rk}
With Einstein-Maxwell system (\ref{EM1}) we get $H^{0,0}=g_0\simeq
S^2\oplus\La^2=T^*\ot T^*$, which correspond to Rainich "already
unified field theory" \cite{R} (so that we don't have bosonic or
fermionic parts, but just tensors). The other cohomology do not sum
(since belong to different bi-grades) but unite and we get that the
only non-vanishing Spencer cohomology of (\ref{EM1}) are:
 \begin{multline*}
H^{0,0}\simeq T^*\ot T^*,\ H^{0,1}\simeq T^*\oplus\La^3T^*,\
H^{1,1}\simeq S^2T^*,\\
H^{0,2}\simeq\R\oplus\Lambda^4T^*,\ H^{1,2}\simeq T^*,\\
H^{0,3}\simeq\Lambda^5T^*,\ H^{0,4}\simeq\Lambda^6T^*,\dots
 \end{multline*}
Notice that for $n=4$ the latter line disappears.
 \end{rk}

 \begin{rk}
Following Rainich \cite{R} Einstein-Maxwell equations are equivalent to the system
 $$
(\op{Ric}^2)_0=0,\ R=0,
 $$
where $\op{Ric}$ is viewed as operator and $L_0=L-\frac14\op{Tr}(L)$ is the
traceless part of an operator $L$. Though relation between the two systems is non-local (but
rather simple integral), involutivity holds for them simultaneously
(however the Spencer cohomology vary, as well as Cartan characters). The latter system,
though more compact\footnote{It contains 10 equations on 10 unknowns,
the same as for Einstein equation (\ref{E}).}, is fully non-linear and is more complicated.
 \end{rk}

\section{Relativity equations of other types}\label{S5}

 \begin{theorem}\label{Gen}
Consider the following PDE system $\E$
 \begin{equation}\label{GenEeq}
G[Ric(g)]=\mathcal{T},\quad L[g,\phi]=0,
 \end{equation}
where $\mathcal{T}=\mathcal{T}[g,\phi]$ is an operator of order 1 in the fields $\phi$ and
the gravitation (metric) $g$,
and $L:C^\infty(M,N_1)\to C^\infty(M;N_2)$ is an (under)determined operator of order $r$
in $\phi$ when $g$ fixed ($\dim N_1\ge\dim N_2$; it is also a differential operator in $g$ of
order $\le r$; we assume $r=1,2$).

Suppose that the Bianchi identity $\op{div}_g(\mathcal{T})=0$ is a differential corollary of
equations $L=0$. Then the system $\E$ is involutive and its only non-trivial Spencer cohomology groups are
 $$
H^{0,0}(\E)=S^2T^*\oplus N_1,\ \ H^{*,1}(\E)=S^2T^*\oplus N_2,\ \
H^{1,2}(\E)=T^*.
 $$
 \end{theorem}

For $r=2$ we have $H^{*,1}(\E)=H^{1,1}$. For $r=1$ we get $H^{*,1}(\E)=H^{1,1}\oplus H^{0,1}$
and the system is of mixed orders in the sense of \cite{KL$_1$}.
It is possible to treat the operator $L$ of mixed (and high) orders $r_1<\dots<r_s$ using the same technique,
but we omit this for transparency of exposition.

Let us notice that the theorem assumption on compatibility is equivalent to $\op{div}_g(\mathcal{T})\in[[L]]$,
where the latter denotes the differential ideal of $L$.
This happens due to (under)determinacy of $L$ and implies that the symbol $\z_{\op{div}_g\mathcal{T}}$
is divisible by $\z_L$.

 \begin{proof}
Let us restrict for shortness of the diagrams\footnote{For other $r$ one shall
change $S^{k+3}\ot V_1$ to $S^{k+3}\ot S^2\oplus S^{k+5-r}\ot N_1$ and have similar changes elsewhere.}
to the case $r=2$ .
Similar to Section \ref{S2} we derive the claim from the bi-complex
 $$
\begindc{\commdiag}[3]
 \obj(14,50)[12]{$0$}
 \obj(38,50)[13]{$0$}
 \obj(65,50)[14]{$0$}
 \obj(92,50)[15]{$0$}
 \obj(-1,40)[21]{$0$}
 \obj(14,40)[22]{$g_{k+3}$}
 \obj(38,40)[23]{$g_{k+2}\ot T^*$}
 \obj(65,40)[24]{$g_{k+1}\ot\La^2$}
 \obj(92,40)[25]{$g_{k}\ot\La^3$}
 \obj(111,40)[26]{$\dots$}
 \obj(-1,30)[31]{$0$}
 \obj(14,30)[32]{$S^{k+3}\ot V_1$}
 \obj(38,30)[33]{$S^{k+2}\ot V_1\ot T^*$}
 \obj(65,30)[34]{$S^{k+1}\ot V_1\ot\La^2$}
 \obj(92,30)[35]{$S^{k}\ot V_1\ot\La^3$}
 \obj(111,30)[36]{$\dots$}
 \obj(-1,20)[41]{$0$}
 \obj(14,20)[42]{$S^{k+1}\ot V_2$}
 \obj(38,20)[43]{$S^{k}\ot V_2\ot T^*$}
 \obj(65,20)[44]{$S^{k-1}\ot V_2\ot\La^2$}
 \obj(92,20)[45]{$S^{k-2}\ot V_2\ot\La^3$}
 \obj(111,20)[46]{$\dots$}
 \obj(-1,10)[51]{$0$}
 \obj(15,10)[52]{$S^{k}\ot T^*$}
 \obj(38,10)[53]{$S^{k-1}\ot T^*\ot T^*$}
 \obj(65,10)[54]{$S^{k-2}\ot T^*\ot\La^2$}
 \obj(92,10)[55]{$S^{k-3}\ot T^*\ot\La^3$}
 \obj(111,10)[56]{$\dots$}
 \obj(15,0)[62]{$0$}
 \obj(38,0)[63]{$0$}
 \obj(65,0)[64]{$0$}
 \obj(92,0)[65]{$0$}
 \mor{12}{22}{} \mor{13}{23}{} \mor{14}{24}{} \mor{15}{25}{}
 \mor{22}{32}{} \mor{23}{33}{} \mor{24}{34}{} \mor{25}{35}{}
 \mor{32}{42}{\footnotesize$\z_{G\oplus L}^{(k+1)}$\,}[\atright,\solidarrow]
  \mor{33}{43}{} \mor{34}{44}{} \mor{35}{45}{}
 \mor{42}{52}{}[\atright,\solidarrow]
 \obj(9,15)[]{{\footnotesize${\z^{(k)}_{\text{div}\oplus 0}}$\ }}
 \mor{43}{53}{} \mor{44}{54}{} \mor{45}{55}{}
 \mor{52}{62}{} \mor{53}{63}{} \mor{54}{64}{} \mor{55}{65}{}
 \mor{21}{22}{} \mor{22}{23}{} \mor{23}{24}{} \mor{24}{25}{} \mor{25}{26}{}
 \mor{31}{32}{} \mor{32}{33}{} \mor{33}{34}{} \mor{34}{35}{} \mor{35}{36}{}
 \mor{41}{42}{} \mor{42}{43}{} \mor{43}{44}{} \mor{44}{45}{} \mor{45}{46}{}
 \mor{51}{52}{} \mor{52}{53}{} \mor{53}{54}{} \mor{54}{55}{} \mor{55}{56}{}
\enddc
 $$
where $V_1=S^2\oplus N_1$ and $V_2=S^2\oplus N_2$.
The columns are exact and all the rows except the top one are exact as well (for all $k$ but $-1$ and $0$).
This claim has to be checked only for the fist column at the place $S^{k+1}\ot V_2$.

The prolonged symbol of the equation has triangular form (we omit the sign of prolongation)
 $$
\z_{G\oplus L}=\begin{bmatrix}\z_G & 0 \\ \z_L^g & \z_L^\phi \end{bmatrix}:
(S^{k+3}\ot S^2)\oplus(S^{k+3}\ot N_1)\to(S^{k+1}\ot S^2)\oplus(S^{k+1}\ot N_2)
 $$
and the next operator is
 $$
\z_{\text{div}\oplus 0}=\begin{bmatrix}\z_{\text{div}} & 0 \\ 0 & 0 \end{bmatrix}:
(S^{k+1}\ot S^2)\oplus(S^{k+1}\ot N_2)\to (S^k\ot T^*)\oplus 0.
 $$
Suppose $\z_{\text{div}\oplus 0}(b_1,b_2)=0$, i.e. $\op{div}(b_1)=0$. We wish to solve
$\z_{G\oplus L}(a_1,a_2)=(b_1,b_2)$ that is equivalent to $\z_G(a_1)=b_1$ and
$\z_L^\phi(a_2)=b_2-\z_L^g(a_1)$. The first equation is solvable by the results of Section \ref{S2}.
The second is solvable by the assumptions of the Theorem.

The claim about Spencer cohomology follows and we conclude symbolic involutivity.
The compatibility condition in $H^{1,2}(\E)$ is precisely the same as before -- Bianchi condition --
and it holds identically due to the assumptions.
 \end{proof}

Let us now specify certain important sub-cases of the above theorem. In some occasions we will
need to assume the Lorentzian signature (more generally: non-definite $g$) and any dimension $n>2$,
while for the rest the metric tensor can be arbitrary. We refer to \cite{K,St}
for physical motivations.

\smallskip

{\bf 1. Pure radiation (null dust).} In this case
$\mathcal{T}=\epsilon\,k^\flat\ot k^\flat$, where $k$ is a null vector field $|k|_g^2=0$,
$k^\flat=\langle k,\cdot\rangle_g$ its $g$-dual and $\epsilon>0$ a scalar function (it can be absorbed
into $k$, but we refrain from doing it).

The compatibility $\op{div}_g\mathcal{T}=0$ is equivalent to the condition
 $$
\nabla_kk=\lambda k,\ \text{ where }\lambda=-\mathcal{L}_k(\epsilon)-\epsilon\op{div}_g(k)
 $$
($\mathcal{L}_v$ is the Lie derivative along $v$).
Thus $k$ is a null pre-geodesic and properly choosing $\epsilon$ (re-parametrization)
we get $\lambda=0$ and the above under-determined equation becomes determined:
the first order operator $L[g,k,\epsilon]$ splits into $L_1\oplus L_2$ with
 $$
L_1=\nabla_kk=0,\qquad L_2= \op{div}_g(k)+\mathcal{L}_k(\log\epsilon)=0.
 $$
Denoting by $\ell=\z_L:T^*\ot(T\oplus\R)\to T\oplus\R$ the symbol of this operator at $(k,\epsilon)$
we get its value on a covector $p$:
 $$
\ell_p(\kappa,e)=
\ell(p\ot(\kappa,e))=(\langle p,k\rangle\kappa,\langle p,\kappa\rangle+\langle p,k\rangle\tfrac{e}\epsilon).
 $$
Thus $\det[\ell_p]=\langle p,k\rangle^n/\epsilon$ is non-zero for $p\not\in\op{Ann}(k)$ and the operator $L$
is indeed determined. As an effect, the equations of pure radiation are involutive.

\smallskip

{\bf 2. Perfect fluid.} Here
$\mathcal{T}=(\epsilon+P)\,U^\flat\ot U^\flat+P\,g$, where $U$ is the particles velocity field, $|U|_g^2=-1$,
and $P$ denotes the pressure. The latter is related to the energy density $\epsilon$ through
an equation of state (constitutive relation,
conservation law etc)\footnote{If this is not imposed, the operator $L$ remains underdetermined.}.
This is often an ODE between $\epsilon$ and $P$, but we will just suppose that $P=P(\epsilon)$.

The compatibility conditions are ($\pi_V$ is the orthogonal projection to the subspace $V$)
 $$
(\epsilon+P(\epsilon)) \nabla_UU+\pi_{U^\perp}(\op{grad}_gP(\epsilon))=0,\qquad
\mathcal{L}_U(\epsilon)+(\epsilon+P(\epsilon))\op{div}_g(U)=0,
 $$
and we take these to be the components $(L_1,L_2)$ of our operator $L=L[g,U,\epsilon]$.
The symbol $\ell=\z_L:T^*\ot(T\oplus\R)\to T\oplus\R$ at $(U,\epsilon)$
has the following we value on a covector $p$:
 $$
\ell_p\left(\left[\begin{array}{c}u\\ e\end{array}\right]\right)=
\left[\begin{array}{cc}(\epsilon+P(\epsilon))\,\langle p,U\rangle\cdot I & \pi_{U^\perp}(p)\,P'(\epsilon)\\
(\epsilon+P(\epsilon))\,p & \langle p,U\rangle\end{array}\right]\cdot
\left[\begin{array}{c}u\\ e\end{array}\right].
 $$
Since $\pi_{U^\perp}(p)=p-\frac{\langle p,U\rangle}{|U|^2_g}U$ we get
 $$
\det\ell_p= (\epsilon+P(\epsilon))^{n-1}\,\langle p,U\rangle^n\,
\left|\begin{array}{cc}I & \Bigl(\frac{p}{\langle p,U\rangle}-\frac{U}{\langle U,U\rangle}\Bigr)\,P'(\epsilon)\\
\frac{p}{\langle p,U\rangle} & 1\end{array}\right|.
 $$
This does not vanish for $p=U$ and hence does not vanish identically. Thus the operator $L$ is determined.
Therefore the equations of perfect fluid are involutive.

\smallskip

{\bf 2a. Specifications: dust etc.} For $P=0$ ($\epsilon>0$) the operator $L=(L_1,L_2)$ with
 $$
L_1=\nabla_UU=0,\qquad L_2= \op{div}_g(U)+\mathcal{L}_U(\log\epsilon)=0.
 $$
has the same form as the one
for the null dust. This case is called the {\em dust\/}. Notice though that the compatibility condition
$L=0$ is obtained without re-normalization in this case since $T=\R\cdot U\oplus U^\perp$
(contrary to the null dust case).

Some other cases -- {\em incoherent radiation\/} $P=\epsilon/3$ and {\em stiff matter\/} $P=\epsilon$
-- provide the other involutive relativity systems.

\smallskip

{\bf 3. Prescribed energy-momentum.} Consider now equation (\ref{E}) with fixed non-zero $\mathcal{T}$.
We change this according to footnote$^{\ref{DeT}}$ to the equation
 \begin{equation}\label{EDT1}
G(\op{Ric}(g))=\phi^*\mathcal{T}
 \end{equation}
with an unknown local diffeomorphism $\phi$. We add the compatibility (conservation law)
 \begin{equation}\label{EDT2}
L[g,\phi]=\op{div}_g(\phi^*\mathcal{T})=0.
 \end{equation}
It is the 2nd order operator in $\phi$ (1st order in $g$) and its symbol
$\z_L=\z_L^\phi:S^2\ot T\to T$ at $\phi\in\op{Diff}_\text{loc}(M)$ equals
 $$
\z_L(p^2\ot \varphi)(q) =|p|^2_g\,\mathcal{T}(\varphi,d\phi(q^\sharp))+
\langle p,q\rangle_g\,\mathcal{T}(\varphi,d\phi(p^\sharp)),\qquad p,q\in T^*,\varphi\in T.
 $$
Clearly if $\mathcal{T}$ is degenerate, then $\z_L(p^2\ot \varphi)=0$ for $\varphi\in\op{Ker}(\mathcal{T})$.
Let us show that for non-degenerate $\mathcal{T}$ and $|p|_g\ne0$ the operator
$\ell_p=\z_L(p^2\ot\cdot):T\to T$ is an isomorphism.

Choose any $\varphi\ne0$. If $\mathcal{T}(\varphi,d\phi(p^\sharp))\ne0$, then $\ell_p(\varphi)(p)=2|p|^2_g\mathcal{T}(\varphi,d\phi(p^\sharp))\ne0$ and so $\ell_p(\varphi)\ne0$.
But if $\mathcal{T}(\varphi,d\phi(p^\sharp))=0$, then
$\ell_p(\varphi)(q)=|p|^2_g\mathcal{T}(\varphi,d\phi(q^\sharp))\ne0$ for a.e. $q\in T^*$ since
$d\phi:T\to T$ is an isomorphism. This shows that $\ell_p$ is injective and hence isomorphism.

Therefore the operator $L$ is determined for non-degenerate $\mathcal{T}$ and so the equation $\E$
(\ref{EDT1})+(\ref{EDT2}) is involutive.
This is a new proof of DeTurck's $C^\oo$ theorem \cite{DT}.

 \com{
In the case of electromagnetic tensor $\mathcal{T}$, as in Section \ref{S4}, the corresponding Maxwell equations
$L=0$ are overdetermined, and so Theorem \ref{Gen} does not apply. However we can use the last trick to include
the electromagnetic source $J$ into the field equations.

Indeed, the source-free equations are $\op{Diff}_\text{loc}(M)$-gauge invariant, and so the equation
with source $J$ can be re-cast into the form (\ref{EM1}) with $\phi^*J$ instead of $J$, where $J\in \mathcal{D}_M$
is a fixed vector field and $\phi\in\op{Diff}_\text{loc}(M)$ is unknown.

The compatibility condition, which we again take for $L=0$, in this case is ...
Its symbol at $\phi\in\op{Diff}_\text{loc}(M)$ equals
 $$
(\z_L)_p(\varphi)=\langle\varphi,p\rangle\,\langle p,J\rangle.
 $$
Therefor for $J\ne0$ the EM .. are NOT involutive.
 }

\section{Conclusion}

In this paper we treated involutivity of several relativity field equations via formal theory of PDEs.
Some classical cases, such as Dirac-Weyl equations, and many modern specifications are not touched,
though can be treated via the same technique.  Some other cases are more delicate.

For instance the scalar minimally coupled matter field equations with the action (here $\phi$ is a function
and $dV_g=\sqrt{|\det g|}dx$ the volume form)
 $$
\mathcal{S}[g,\phi]=\int_M[R-|d\phi|_g^2-m^2\phi^2]\,dV_g
 $$
have block-form symbol operator $\z_{G\oplus L}$ (with
$\mathcal{T}=d\phi\ot d\phi-\tfrac12|d\phi|_g^2-\tfrac12m^2\phi^2$; $L[\phi]=\square\phi-m^2\phi$),
and so are involutive by Theorem \ref{Gen}.

Non-minimally coupled field equations with the action
 $$
\mathcal{S}[g,\phi]=\int_M[f(\phi)R-|d\phi|_g^2-m^2\phi^2]\,dV_g
 $$
have the general symbol operator, and our methods do not apply directly.
However the system can be re-written with triangular symbol, and then methods of
Section \ref{S5} show it is involutive\footnote{For almost all functions $f(\phi)$,
in particular for the important case $f(\phi)=1-\xi\phi^2$.}.

Let us deduce several corollaries of the involutivity. They are
based on the Cartan-K\"ahler theorem claiming that a formally
integrable analytic system has local solutions. Since involutivity
implies formal integrability, we conclude

 \begin{theorem}
Let $j^k_0g$ be a jet of metric ($1<k<\infty$), which satisfies
$(k-2)$-jet of the vacuum Einstein equations (\ref{E})$_{\mathcal{T}=0}$.
Then there exists a local analytic solution $g$ of this equation
with the prescribed jet $j^k_0g$ at the point $0\in M$.

In particular, if a Riemann tensor $\op{Riem}_0$ at the point is
given, which satisfies the obvious algebraic compatibility
conditions with a metric $g_0\in S^2T^*_0M$ through (\ref{E}),
then there exists an analytic solution to the vacuum Einstein
equations with the given initial data $(g_0,\op{Riem}_0)$.
 \end{theorem}

This is a variation on Gasqui's theorem \cite{Ga}. For non-vacuum field equation (\ref{E}),
one similarly deduce DeTurck theorem \cite{DT} on solvability in analytic category, provided
that the tensor $\mathcal{T}$ is non-degenerate.
Turning now to Einstein-Maxwell equation (\ref{EM1}) or (\ref{EM2})
we arrive at

 \begin{theorem}
Let $j^k_0g,j^k_0F$  be $k$-jets of a metric and an analytic 2-form,
which are related by (jets of) source-free Einstein-Maxwell equation
(\ref{EM1})$_{J=0}$. Then there exists a local analytic solution
$(g,F)$ of this equation with the prescribed jet $(j^k_0g,j^k_0F)$.

In particular, if a metric $g_0\in S^2T^*_0M$, a Riemann tensor $\op{Riem}_0$
and a 2-form $F_0\in\La^2T^*_0M$ at the point $0\in M$ are
given, which satisfy the algebraic compatibility conditions through
the first equation of (\ref{EM1}), then there exists an
analytic solution to the source-free Einstein-Maxwell equations with
the given initial data $(g_0,\op{Riem}_0,F_0)$.
 \end{theorem}

Finally we can prove local solvability (with prescribed Cauchy data) of other field equations
considered in Section \ref{S5} (pure radiation, perfect fluid, dust etc).

 \begin{theorem}
Let $j^k_0g,j^k\phi_0F$ be $k$-jets of a metric and additional fields,
related by (jets of) equation (\ref{GenEeq}).
Suppose that both $\mathcal{T}$ and $L$ are analytic.
Then there exists a local analytic solution
$(g,\phi)$ of this equation with the prescribed jet $(j^k_0g,j^k_0\phi)$.
 \end{theorem}

By similar reasons we have local analytic solutions with any admissible Cauchy data
to minimally and non-minimally coupled fields equations etc.

\appendix
\section{Some technicalities}

Here we collect some tedious verifications.
Let us begin by checking the fact that conservation law for (\ref{E})
with fixed $\mathcal{T}$ vanishes only for the vacuum case.

 \begin{lem}
Let $\mathcal{T}\in S^2T^*$ be a constant tensor.
Equation $\op{div}_g(\mathcal{T}_\Lambda)=0$
from Section \ref{S2} is non-trivial only for $\mathcal{T}=0$.
 \end{lem}

 \begin{proof}
We shall calculate symbol of the operator $\op{div}_g(\mathcal{T}_\Lambda)$ with respect to $g$ (which varies),
and regardless $\mathcal{T}$ (which stays constant). Metric $g$ enters via Christoffel coefficients, namely
in local coordinates:
 $$
\op{div}_g(\mathcal{T}_\Lambda)_i=(\mathcal{T}_\Lambda)_{i;j}^j=
\p_{x_j}(\mathcal{T}_\Lambda)^j_i+\Gamma_{kj}^j(\mathcal{T}_\Lambda)_i^k-\Gamma_{ij}^k(\mathcal{T}_\Lambda)_k^j.
 $$
This yields the following formula for the symbol operator
(we $g$-lift $\mathcal{T}_\Lambda$ to the operator field)
 $$
\z_{\op{div}_g\mathcal{T}_\Lambda}(p\ot q^2)
=\tfrac12(|q|^2_g\,\mathcal{T}_\Lambda p-\langle\mathcal{T}_\Lambda q,q\rangle_g\,p).
 $$
This vanishes if and only if the operator $\mathcal{T}_\Lambda$ is scalar. Since $\mathcal{T}$
(as a covariant tensor) is fixed ($g$-independent),
this implies $\mathcal{T}_\Lambda=-\Lambda\,g$ and $\mathcal{T}=0$.
 \end{proof}

Next we study exactness of the complexes from Section \ref{S2}.

 \begin{lem}
The operator $\z_{\mathfrak{d}}^{(k)}:S^{k+1}\ot S^2\to S^k\ot T^*$ is epimorphic.
 \end{lem}

 \begin{proof}
We have:
 $$
\z_{\mathfrak{d}}^{(k)}\bigl(\tfrac1{k+1}p^{k+1}\ot pq\bigr)=
\tfrac{|p|_g^2}2\,p^k\ot q.
 $$
Since the latter vectors generate $S^k\ot T^*$ the claim follows.
 \end{proof}

  \begin{lem}
In the diagrams from Section \ref{S2} all columns are exact.
 \end{lem}

 \begin{proof}
Formula (\ref{div}) implies that the following sequences are complexes:
 $$
0\to g_{k+3}\to S^{k+3}\ot S^2\stackrel{\z_{\op{Ric}}^{(k+1)}}\longrightarrow S^{k+1}\ot S^2
\stackrel{\z_{\mathfrak{d}}^{(k)}}\to S^{k}\ot T^*\to 0.
 $$
Since the previous discussion implies that the only possible cohomology can be supported at
the term $S^{k+1}\ot S^2$, the exactness can be verified by
dimensional reasons as in \cite{DT,Ga}.

We will show another approach, but restrict for simplicity of formulas to the case $k=0$
(prolongations can be treated similarly).

The space $\op{Ker}(\z_{\op{div}})$ is generated by vectors $p\ot q^2$ and
$|q|_g^2\,p\ot p^2-2|p|^2_g\,q\ot qp$ with $p\perp q$. For $n>2$ the operator $G$ is
invertible: $G^{-1}h=h-\frac1{n-2}\op{Tr}_g(h)g$. Thus
 $$
\op{Ker}(\z_{\mathfrak{d}})=(1\ot G)^{-1}\op{Ker}(\z_{\op{div}})=
\bigl\langle p\ot q^2-\tfrac{|q|_g^2}{n-2}p\ot g,
p\ot q^2-\tfrac{|q|_g^2}{|p|^2_g}p\ot p^2+2q\ot qp\,|\,p\perp q\bigr\rangle
 $$
We need to show that the generating elements belong to $\op{Im}(\z_{\op{Ric}}^{(1)})$.
This follows from exact formulae
 \begin{multline*}
\z_{\op{Ric}}^{(1)}\Bigl(
\tfrac2{|p|_g^2}\,p^2q\ot pq+\tfrac4{|h|_g^2}\,q^2h\ot ph
-\tfrac2{|h|_g^2}\,ph^2\ot q^2-\tfrac{2|q|^2_g}{|p|^2_g|h|^2_g}\,p^2h\ot ph
+\tfrac2{3|p|^2_g}\,p^3\ot q^2\Bigr)\\
=p\ot q^2-\tfrac{|q|_g^2}{|p|^2_g}p\ot p^2+2q\ot qp,
 \end{multline*}
 \begin{multline*}
\z_{\op{Ric}}^{(1)}\Bigl(
\tfrac2{n-2}(g\cdot p)\ot q^2-\tfrac4{n-2}\,g\bar\ot(q^2\ot p)+
\tfrac2{|p|_g^2}\,p^2q\ot pq+\tfrac4{|h|_g^2}\,q^2h\ot ph\\
-\tfrac2{|h|_g^2}\,ph^2\ot q^2-\tfrac{2|q|^2_g}{|p|^2_g|h|^2_g}\,p^2h\ot ph
\Bigr)=p\ot q^2-\tfrac{|q|^2}{n-2}\,p\ot g,
 \end{multline*}
where $p,q,h$ are mutually orthogonal elements (with non-zero length; since
both kernel and image are closed spaces this does not restrict generality),
and $g\bar\ot(q^2\ot p)=\sum\epsilon_i (q^2e_i)\ot(p\,e_i)$ for $g=\sum\epsilon_ie_i^2\in S^2T^*$.
 \end{proof}

Finally the following statements was used in Section \ref{S5}.

 \begin{lem}
Consider a vector bundle $\pi$ over $M$ with fiber -- vector space $N$. Let
$F:C^\infty(\pi)\to C^\infty(M,N)$ be a determined differential operator of order $k$.
Then the corresponding equation
$\E=\op{Ker}(\Delta_F)$ is involutive and its only non-zero Spencer cohomology groups are
$H^{0,0}=H^{k-1,1}=N$.
 \end{lem}

 \begin{proof}
This is a folklore result, and it is difficult to find the reference. One of the proofs
is the diagram chase by the bi-complex with vertical complex being ($g=\oplus g_i$
denotes the symbolic system corresponding to $\E$)
 $$
0\to g\to ST^*\ot N\stackrel{\z_{F}}\longrightarrow ST^*\ot N\to0
 $$
and the horizontal one being the Spencer complex generated by the first column.

Another possibility is to view this as a direct corollary of Theorem A of \cite{KL$_4$},
which concerns reductions of Cohen-Macaulay systems (ideal of the determined equations is
clearly Cohen-Macaulay). This theorem states that for a subspace $V^*\subset T^*$ transversal
to the characteristic variety $\op{Char}(\E)\subset \mathbb{P}T^*$ (both complexified) the reduced
symbolic system $g\cap SV^*\ot N$ has the same Spencer cohomology as the symbolic system $g$.
It suffices to apply this to any generic 1-dimensional line $V^*$ (for ODEs the claim is obvious)
since $\op{codim}\op{Char}(\E)=1$.
 \end{proof}


\end{document}